\documentclass{amsart}

\usepackage{mathrsfs}

\usepackage[latin1]{inputenc}
\usepackage{amsfonts}
\usepackage{amsmath}
\usepackage{amsthm}
\usepackage{amssymb}
\usepackage{latexsym}
\usepackage{enumerate}

\newtheorem{theorem}{Theorem}
\newtheorem{lemma}[theorem]{Lemma}

\newtheorem{definition}[theorem]{Definition}

\numberwithin{equation}{section}

\begin{document}

\newcommand{\cc}{\mathfrak{c}}
\newcommand{\N}{\mathbb{N}}
\newcommand{\C}{\mathbb{C}}
\newcommand{\Q}{\mathbb{Q}}
\newcommand{\R}{\mathbb{R}}
\newcommand{\T}{\mathbb{T}}
\newcommand{\seqn}{\{0,1\}^n}
\newcommand{\seqm}{\{0,1\}^m}
\newcommand{\In}{I_n}
\newcommand{\im}{I_m}
\newcommand{\st}{*}
\newcommand{\PP}{\mathbb{P}}
\newcommand{\lin}{\left\langle}
\newcommand{\rin}{\right\rangle}
\newcommand{\SSS}{\mathbb{S}}
\newcommand{\forces}{\Vdash}
\newcommand{\dom}{\text{dom}}
\newcommand{\osc}{\text{osc}}
\newcommand{\F}{\mathcal{F}}
\newcommand{\A}{\mathcal{A}}
\newcommand{\K}{\mathcal{K}}
\newcommand{\B}{\mathcal{B}}
\newcommand{\I}{\mathcal{I}}
\newcommand{\CC}{\mathcal{C}}

\author{Piotr Koszmider}
\address{Institute of Mathematics, Polish Academy of Sciences,
ul. \'Sniadeckich 8,  00-656 Warszawa, Poland}
\email{\texttt{piotr.koszmider@impan.pl}}

\title{A non-diagonalizable pure state}

\maketitle

\begin{abstract} We construct a pure state  on the C*-algebra
$\B(\ell_2)$ of all bounded linear operators on $\ell_2$ 
which is not diagonalizable, i.e.,
it is not of the form $\lim_u\langle T(e_k), e_k\rangle$ for any
orthonormal basis $(e_k)_{k\in \N}$ of $\ell_2$ and an ultrafilter $u$ on $\N$.
This constitutes a counterexample to Anderson's conjecture without additional hypothesis and
 improves  results of C. Akemann, N. Weaver, I. Farah and I. Smythe
  who constructed such states making additional set-theoretic assumptions.
  
  It follows from results of J. Anderson and the positive solution to the Kadison-Singer problem
due to A. Marcus, D. Spielman, N. Srivastava that 
the restriction of our pure state to any atomic masa $D((e_k)_{k\in \N})$ of
diagonal operators with  respect to an orthonormal basis $(e_k)_{k\in \N}$ is not multiplicative
on $D((e_k)_{k\in \N})$. 
\end{abstract}

\section{Introduction}

Recall that a pure state on a C*-algebra is a positive linear functional of norm one, i.e., a state,
 which is not a convex combination of other states. Pure states on
the algebras of all operators on
finite dimensional Hilbert spaces $\ell_2(n)$ for $n\in \N$ are known all to
be vector states, i.e., of the form $\phi(T)=\langle T(v), v\rangle$, where $v\in \ell_2(n)$
is a unit vector. Vector states are also pure states in the  case 
of the algebra $\B(\ell_2)$ of all linear bounded operators on an infinite dimensional
Hilbert space $\ell_2$. 

There are many other pure states on $\B(\ell_2)$ whose existence
is usually proved by means of the Hahn-Banach theorem starting from a pure state
on a maximal abelian self-adjoint subalgebra  (masa) of $\B(\ell_2)$.
If the masa is atomic, that is of the form $D((e_k)_{k\in \N})$ of
all diagonal operators with respect to an orthonormal basis $(e_k)_{k\in \N}$,
then the general form of the initial pure state $\phi$ for $T\in D((e_k)_{k\in \N})$ is 
$$\lim_u\langle T(e_k), e_k\rangle \leqno (\rm D)$$
 where $u$ is an 
ultrafilter\footnote{$\lim_{ u}z_k=z$
for $z_k, z\in \C$ means that for every $\varepsilon>0$ the set $\{k\in \N: |z_k-z|<\varepsilon\}$
is in the ultrafilter $u$. } on $\N$. 
J. Anderson showed in \cite{anderson-extreme} that
(D) defines a pure state on the entire $\B(\ell_2)$ and  conjectured
in what became known as Anderson's conjecture 
(\cite{anderson}) that every pure state on $\B(\ell_2)$ is of the above form
for some orthonormal basis $(e_k)_{k\in \N}$ of $\ell_2$ and an ultrafilter 
$u$ on $\N$.  Our main result (Theorem \ref{main-theorem}) is the construction
of a pure state that is non-diagonalizable, that is a counterexample to Anderson's conjecture.

Much of the research concerning the relations between pure states on $\B(\ell_2)$ and
pure states on masas of $\B(\ell_2)$ has been motivated by a seminal paper 
\cite{ks} of Kadison and Singer. The positive solution of one of the problems stated in the paper and known
as the Kadison-Singer problem  due to A. Marcus, D. Spielman, N. Srivastava 
implies that a non-diagonalizable pure state on $\B(\ell_2)$ necessarily cannot have
multiplicative restriction to any atomic masa (because such restrictions extend to 
pure states on $\B(\ell_2)$ of the form (D) but the extensions are unique by the
positive solution to the Kadison-Singer problem).

Another problem from the paper \cite{ks}  is whether any pure state on $\B(\ell_2)$
has a multiplicative restriction to some masa of $\B(\ell_2)$. Having
multiplicative restriction in this case is equivalent to 
having the restriction equal to a pure state on the masa.
In \cite{ch-masas}
C. Akemann and N. Weaver provided a negative solution to this problem assuming 
the continuum hypothesis {\sf CH}.
This, in particular, already showed that Anderson's conjecture is consistently false,
but as suggested in \cite{ch-masas} it could still be consistent that 
any pure state on $\B(\ell_2)$ has a multiplicative restriction to a masa.
This additional hypothesis in the case
of Anderson's conjecture was weakened  to {\sf MA}
(\cite{weaver-book}) or  to ${\mathsf{cov}}(\mathcal M)=\mathfrak c$
or to $ \mathfrak d\leq \mathfrak p^*$ (12.5 of \cite{farah-book}), or to
another one in \cite{smythe}.

Our counterexample to Anderson's   conjecture  shows 
that the additional hypothesis in the result of Akemann and Weaver
is not needed when we limit ourself to atomic masas.  However, we do not know if
our non-diagonalizable pure state can have a multiplicative restriction to a non-atomic masa.
An improvement of a result from \cite{ks} due to J. Anderson from \cite{anderson-extensions} says that any
pure state on a non-atomic masa has many extensions to pure states on $\B(\ell_2)$.
So we can say that our pure state is not ``determined" by any pure restriction to any masa, i.e.,
either the restriction is not pure or if it is pure it does not uniquely extends to
our pure state.

Our construction is entirely different than that
of Akemann and Weaver which used properties of separable C*-subalgebras of $\B(\ell_2)$ and a well-ordering of all masas in the first uncountable
type $\omega_1$ based on the continuum hypothesis to approximate
the desired pure state with separable fragments.
  Let us describe the main idea of our construction here. In a sense,
  instead of using separable approximations we obtain the desired pure state by
  approximating it with finite dimensional fragments. Let
$\{0,1\}^m$ denote the set of all sequences of zeros and ones of length $m\in \N$ and
let $\{0,1\}^\N$ denote the set of all infinite sequences of zeros and ones.
We fix a function $d:\N\rightarrow \N$  which will be specified later and identify
 $\ell_2$ with
$$\bigoplus_{m\in \N}\bigotimes_{\sigma\in \{0,1\}^m}\ell_2(d(m)).\leqno (\rm I_1)$$
This can be done by considering a partition of $\N$ into finite sets of sizes 
$d(m)^{(2^m)}$ for $m\in \N$.
Recall that there is a canonical isomorphism
$$\B\big(\bigotimes_{\sigma\in \{0,1\}^m}\ell_2(d(m))\big)\equiv
\bigotimes_{\sigma\in \{0,1\}^m}\B(\ell_2(d(m))).\leqno ({\rm I_2})$$
If for every $\sigma\in \{0,1\}^m$ we choose any non-zero projection $P_\sigma\in \B(\ell_2(d(m)))$
and $I$ denotes the identity, then
$$\prod_{\sigma\in \{0,1\}^m} (I\otimes ... \otimes \overset{\sigma}P_\sigma\otimes ... \otimes I)
=\bigotimes_{\sigma\in \{0,1\}^m}P_\sigma.\leqno (\rm P)$$
is a projection that is dominated by any of the projections 
$(I\otimes ... \otimes \overset{\sigma}P_\sigma\otimes ... \otimes I)$.
So for any $\alpha\in \{0,1\}^\N$ and any choice 
$v^\alpha=(v_m^\alpha)_{m\in \N}\in \prod_{m\in\N}(\ell_2(d(m))\setminus\{0\})$
  we can define
 rank one projections $R_{v_m^\alpha}:\ell_2(d(m))\rightarrow \ell_2(d(m))$ 
onto the direction of $v_m^\alpha$ and distribute them along 
the restrictions $\alpha|m=\alpha|\{0, ..., d(m)-1\}\in\{0,1\}^m$ for $m\in\N$ defining
$$P_{\alpha, v^\alpha}=\bigoplus_{m\in \N} \big(I\otimes ... 
\otimes \overset{\alpha|m}{R_{v_m^\alpha}}\otimes ... \otimes I\big).$$
 Under the identifications (${\rm I_1}$) and (${\rm I_2}$)  
 the operator $P_{\alpha, v^\alpha}$  is a projection in $\B(\ell_2)$. It follows
 from (P)  that
for any choices $v^\alpha$ for  $\alpha\in \{0,1\}^\N$
 any finite product formed by the projections $(P_{\alpha, v_\alpha}: \alpha\in \{0,1\}^\N)$  
 dominates a nonzero projection 
 because eventually $\alpha_1|m, ...,  \alpha_n|m\in \{0,1\}^m$ are all distinct
 if $\alpha_1, ..., \alpha_n\in\{0,1\}^\N$ are distinct. This guarantees that for
 any choices $v^\alpha$ for $\alpha\in \{0,1\}^\N$
there is a pure state $\phi$ on $\B(\ell_2)$ such that $\phi(P_{\alpha, v^\alpha})=1$
for all $\alpha\in \{0,1\}^\N$.

To make sure that $\phi$ is not diagonalized 
by any orthonormal basis we need to show that there is a constant $0<c<1$ such that 
for every orthonormal basis
$(e_k)_{k\in \N}$ of $\ell_2$ there is $\alpha\in \{0,1\}^\N$ such that 
$$|\langle P_{\alpha, v^\alpha}(e_k), e_k\rangle|<c\leqno ({\rm ND})$$
for every $k\in \N$.  
To obtain the above property we manipulate the choice
of $v^\alpha$. Here we exploit the fact that 
if $f(d)$  points on $d$-dimensional real sphere form an $\varepsilon$-net on the sphere
for $\varepsilon<1$,
then $f(d)$ must grow exponentially in the dimension $d\in \N$. 
Using this with the choice of $d$ satisfying for each $m\in \N$
(i) $d(m)\geq 2^7$ , (ii)
$32m^2(d(m)^{(2^m)})^2d(m)^{(2^m-1)}<(100/91)^{d(m)}$ we can obtain $v^\alpha$
satisfying (ND) for $c=19/20$ and a fixed orthonormal basis
$(e_k)_{k\in \N}$ of $\ell_2$. As there are as many orthogonal bases in $\ell_2$ as
elements  $\alpha\in\{0,1\}^\N$, we can make sure that $\phi$ is not diagonalized by any basis.

The structure of the paper is as follows. In the second section we 
discuss the preliminaries including the above mentioned tensor products
of finite dimensional Hilbert spaces and the exponential growth of the above mentioned
function. In the third section we construct the required family of projections
(Theorem \ref{family}) and include the final argument.
The main result is Theorem \ref{main-theorem}. The last section contains additional remarks.

The notation should be standard. When $X$ is a set, then $\ell_2(X)$ denotes
the Hilbert space whose orthonormal basis is labeled by elements of $X$.
All norms are $\ell_2$-norms or operator norms on Hilbert spaces.
$|X|$ denotes the cardinality of a set and $|z|$ denotes the absolute value of a
complex number $z$,  it should be always clear
from the context which meaning of $|\ |$ is used. We also often identify
$n\in\N$ with the set $\{0, ..., n-1\}$. For sets $A, B$ by $B^A$ we mean
the set of all functions from $A$ into $B$. The restriction $\sigma=x|m$
of an infinite sequence $x\in\{0,1\}^\N$ for $m\in \N$ is 
a sequence $\sigma\in \{0,1\}^m$ of length $m$ such that $\sigma(k)=x(k)$
for all $k<m$.

The author would like to thank  
 C. Akemann, I. Farah, P. Wojtaszczyk  for valuable comments which were
 used to improve the previous version of the paper.

\section{Preliminaries}

\subsection{Projections and  the inner product}

\begin{lemma}\label{equality}Suppose that $P$ is an orthogonal projection in $\B(\ell_2)$
and $x\in\ell_2$. Then
$$\langle P(x), x\rangle=\|P(x)\|^2.$$
\end{lemma}
\begin{proof} Using the facts that $P=P^2=P^*$ we obtain
 $\langle P(x), x\rangle=
\langle P^2(x), x\rangle=
\langle P(x), P(x)\rangle=\| P(x)\|^2$.
\end{proof}

\begin{lemma}\label{big-projections} Suppose that $(e_k)_{k\in \N}$ is an orthonormal basis of $\ell_2$
and $F\subseteq \ell_2$ is an $n$-dimensional linear subspace of $\ell_2$. 
Let $\varepsilon >0$. There is $X\subseteq \N$ of cardinality not bigger than $n^2/\varepsilon$
such that $\|P_F(e_k)\|^2<\varepsilon$ for every $k\in \N\setminus X$.
\end{lemma}
\begin{proof} Let $\{e_0', ..., e_{n-1}'\}$ be an orthonormal basis of $F$.
We have $1=\|e_j'\|^2=\Sigma_{k\in \N}|\langle e_j', e_k\rangle|^2$ for each $j<n$.
So there are $A_j\subseteq \N$ of cardinality not bigger than  $n/\varepsilon$ such that
$|\langle e_j', e_k\rangle|^2\leq \varepsilon/n$ for every $k\in \N\setminus A_j$.
Let $X=\bigcup_{j<n} A_j$. Then $|X|\leq n^2/\varepsilon$ and for $k\in \N\setminus X$ we have 
$$\|P_F(e_k)\|^2=\langle \Sigma_{j<n}\langle e_k, e_j'\rangle e_j',
 \Sigma_{j<n}\langle e_k, e_j'\rangle e_j'\rangle=
\Sigma_{j<n}\langle \langle e_j', e_k \rangle e_j',  \langle e_j', e_k \rangle e_j'\rangle=$$
$$\Sigma_{j<n}|\langle e_j', e_k \rangle|^2 \leq n\varepsilon/n=\varepsilon.$$
\end{proof}

\subsection{Obtaining an inclined vector}

The purpose of this subsection is to prove Lemma \ref{inclined-vector} which roughly says that
there is an absolute constant  such that if in $d$-dimensional
Hilbert space we have less then ``exponentially in $d$"-many  directions, then
there is another direction whose inclination
to all the original ones is at least the constant. 

\begin{lemma}\label{net} Suppose that $X$ is a collection of unit vectors
in $\R^d$
for $d\geq 2^7$,
such that for every unit vector $y\in \R^d$ there is
$x\in X$ with $\|x-y\|\leq 9/10$. Then the cardinality
of $X$ is at least $(100/91)^d/2$.
\end{lemma}
\begin{proof} Let $B_r(a)$ denote the ball of radius $r>0$ with the center $a\in \R^d$.
Let $V_d(r)$ denote the $d$-dimensional volume of $B_r(a)$ for  any $a\in \R^d$.
Recall that  $V_d(r)$ is equal to $r^dV_d(1)$ which follows from
the formula for integration by substitution with the substitution
sending $a\in \R^d$ to $ra$.

The hypothesis on $X$ implies that  the sphere in $\R^d$ is covered by
$\bigcup\{B_{9/10}(x): x\in X\}$. So whenever $99/100\leq \|y\|\leq 1$ for $y\in \R^d$, then
there is $x\in X$ such that $d(x, y)\leq d(x, y/\|y\|)+1/100\leq 9/10+1/100=91/100$
 and so set
$\bigcup\{B_{91/100}(x): x\in X\}$ covers 
$B_1(0^d)\setminus B_{99/10}(0^d)$, where $0^d$ denotes the origin in $\R^d$.  

The latter set has volume
$V_d(1)-V_d(99/100)=(1-(99/100)^d)V_d(1)$ and the union which covers it
has volume not bigger than $|X|(91/100)^dV_d(1)$.
It follows that $(1-(99/100)^d)V_d(1)\leq |X|(91/100)^dV_d(1)$ and so
$(100/91)^d({1-(99/100)^d})\leq |X|$. As $(99/100)^{(2^7)}\approx0,276251668$
we have that $(99/100)^d\leq1/2$ for $d\geq 2^7$ and so
$|X|\geq (100/91)^d/2$ for such $d$s, as required.
\end{proof}

The above  argument is a version of well known fact concerning $\varepsilon$-nets 
of the $n$-dimensional ball,
e.g. Proposition 15.1.3 of \cite{approx}.

\begin{lemma}\label{product-distance} Suppose that $d\in \N\setminus\{0\}$
and $x, y\in \C^d$ are unit vectors.  
Suppose that $\varepsilon>0$ and 
$\|x\pm y\|, \|x\pm iy\|\geq\varepsilon$.
Then $$|\langle x, y\rangle|\leq \sqrt2(1-\varepsilon^2/2).$$
\end{lemma}
\begin{proof} Let $\alpha\in \{1, -1, i, -i\}$.
By the parallelogram law $2(\|v\|^2+\|w\|^2)=\|v+w\|^2+\|v-w\|^2$
we conclude that $\|x+\alpha y\|^2=4-\|x-\alpha y\|^2$ and
$\|x+i\alpha y\|^2=4-\|x-i\alpha y\|^2$. Using the above and the
polarization identity 
$\langle u, v \rangle={1\over 4}(\|u+v\|^2-\|u-v\|^2- i\|u-iv\|^2+i\|u+iv\|^2)$ we obtain

$$|\langle x, y\rangle|=|\langle x, \alpha y\rangle|=
\sqrt{(1-\|x-\alpha y\|^2/2)^2+(1-\|x-i\alpha y\|^2/2)^2}\leq \sqrt2(1-\varepsilon^2/2)$$
for $\varepsilon^2/2, \|x-\alpha y\|^2/2, \|x-i\alpha y\|^2/2\leq 1$ 
since the real variable function $(1-\beta)^2$ is decreasing
below $\beta=1$ and we have $\varepsilon^2/2\leq \|x-\alpha y\|^2/2, \|x- i\alpha y\|^2/2$.
 The rest of the proof consist of noting that
under our hypothesis that $\|x\|=\|y\|=1$ there is $\alpha \in \{1, -1, i, -i\}$
such that $\|x-\alpha y\|, \|x-i\alpha y\|\leq \sqrt 2$.

By rotating the sphere in $\C^d$ we may assume that
 $y=(1, 0,  ..., 0)$. 
For any $x_1\in \C$ such that $|x_1|\leq 1$ there
is $\alpha\in \{1, -1, i, -i\}$  such
that $|x_1-\alpha|, |x_1-i \alpha|\leq 1$.
Now 
$$\|x-\alpha y\|^2=|x_1-\alpha|^2+\sum_{1<k<d}|x_k|^2\leq 2.$$
$$\|x- i \alpha y\|^2=|x_1- i \alpha|^2+\sum_{1<k<d}|x_k|^2\leq 2.$$

\end{proof}

Let us note that considering the points $\pm iy$ in the above lemma is necessary in the
complex case as already for $d=1$ we have $\|i+1\|=\|i-1\|=\sqrt 2$ but $i$ and $1$ are not inclined, i.e., 
$|\langle i, 1\rangle|=1$
as $i$ and $1$ lie on the same ``complex straight line" as they are linearly dependent over $\C$.

\begin{lemma}\label{inclined-vector}
Suppose that $d, n\in \N$ satisfy $d\geq 2^7$ and $n<(100/91)^d/8$
and that $X=\{x^j: j<n\}$ is a collection of  vectors in $\C^d$. Then there
is a unit vector $x\in \C^d$ such that $|\langle x_j, x\rangle|\leq(9/10)\|x_j\|$
for every $j<n$. In particular $\|R_x(x_j)\|^2\leq (9/10)\|x_j\|^2$
for every $j<n$, where $R_x$ is the orthogonal projection onto the direction of $x$.
\end{lemma}
\begin{proof} First assume that all $x_j$s are unit vectors.
Identifying $\R^2$ with $\C$ we 
can consider $Y(l)=\{y^j(l): j<n\}\subseteq \R^{2d}$ for $l\in \{1, 2, 3, 4\}$,
satisfying 
$$x^j_k=y^j_{2k}(1)+iy^j_{2k+1}(1),$$
$$-x^j_k=y^j_{2k}(2)+iy^j_{2k+1}(2),$$
$$ix^j_k=y^j_{2k}(3)+iy^j_{2k+1}(3),$$
$$-ix^j_k=y^j_{2k}(4)+iy^j_{2k+1}(4)$$
for all $k<d$ and $j<n$ and $1\leq l\leq 4$. It is clear that $y_j(l)$ are unit
vectors for all $j<n$ and $1\leq l\leq 4$.

As $|Y(1)\cup Y(2)\cup Y(3)\cup Y(4)|=4n<(100/91)^{d}/2<(100/91)^{2d}/2$ Lemma \ref{net} implies
that there is a unit $z\in \R^{2d}$ such that
$\|z-y^j(l)\|\geq 9/10$ for all $j<n$ and $1\leq l\leq 4$.

Consider $x\in \C^d$ whose coordinates are complex numbers whose
real and imaginary parts are formed from the $2d$ real coordinates of $z$, i.e.,
$${x_k}={z_{2k}}+i{z_{2k+1}}$$
for any $k<d$. It is clear that $x$ is a unit vector. We have
$$\|x-{x^j}\|=\sqrt{\sum_{k<d}  |x_k-x^j_k|^2  }=$$
$$\sqrt{\sum_{k<d}  |{z_{2k}}+i{z_{2k+1}}-y^j_{2k}(1)-iy^j_{2k+1}(1)|^2  }=$$
$$=\sqrt{\sum_{k<d}  \big(|{z_{2k}}-y^j_{2k}(1)|^2+|{z_{2k+1}}-y^j_{2k+1}(1)|^2\big)  }=
\|z-{y^j}(1)\|$$
and analogously  $\|x+{x^j}\|=\|z-{y^j}(2)\|$,
$\|x-i{x^j}\|=\|z-{y^j}(3)\|$ and $\|x+i{x^j}\|=\|z-{y^j}(4)\|$. So
$\|x-x^j\|, \|x-i{x^j}\|, \|x+x^j\|, \|x+i{x^j}\|\geq 9/10$ for all $j<n$.

It follows from Lemma \ref{product-distance} that for any $j<n$ we have
$|\langle x, x^j\rangle|\leq \sqrt2(1-(9/10)^2/2)$.
We also have $\sqrt2(1-(9/10)^2/2)\leq \sqrt{9^2\over 6^2}{119\over 200}\leq
{178,5\over 200}\leq {180\over200}= 9/10$, so $|\langle x, x^j\rangle|\leq 9/10$
for every $j<n$.

If $x_j$ have arbitrary norms, we have 
$$|\langle x, x_j\rangle|=\|x_j\|\langle x, x_j/\|x_j\|\rangle
\leq (9/10)\|x_j\|.$$
 Also by Lemma \ref{equality}
$$\|R_x(x_j)\|^2=\langle R_x(x_j), x_j\rangle=\langle {\langle x_j, x\rangle}x, x_j\rangle
=\langle x_j, x\rangle\overline{ \langle x_j, x\rangle}
=|\langle x_j, x\rangle|^2\leq$$
$$\leq(9/10)^2\|x_j\|^2\leq (9/10)\|x_j\|^2.$$
\end{proof}

\subsection{Obtaining inclined intersecting subspaces in tensor products}\label{tensor}

For sets $A, B$ as usual $B^A$ denotes the set of all functions from $A$ to $B$.
$\{(a, b)\}$ will stand for a function whose domain is $\{a\}$ and which assumes
value $b$ at $a$. So any $t\in B^A$ can be written uniquely as $t=s\cup\{(a, b)\}$,
where $s\in B^{A\setminus\{a\}}$. 
We will view the Hilbert
space $\ell_2(B^A)$ as the tensor product of Hilbert spaces
$\bigotimes_{a\in A} \ell_2(B^{\{a\}})$, 
where $\langle \otimes_{a\in A}x_a, \otimes_{a\in A}y_a\rangle=
\prod_{a\in A} \langle x_a, y_a\rangle$ (\cite[6.3.1]{murphy}). This notation
 will allow us to handle many-fold tensor 
products with precision and a relatively modest amount of indices.
For example $e_{\{a, b\}}\otimes e_s=e_{s\cup\{(a, b)\}}=e_s\otimes e_{\{a, b\}}$ and we do not
need to worry about the order of factors in tensor products of Hilbert spaces. However in the
case of tensors of operators we will be using a more standard notation
$S\otimes ... 
\otimes \overset{a}{T}\otimes ... \otimes S$ to indicate  with a letter above $T$ at which coordinate
we put the operator $T$.
Recall that $R_v$ denotes the rank one orthogonal projection onto
the direction of a nonzero vector $v$.

\begin{definition}\label{tensor-projection} Suppose that $A, B$ are nonempty sets
and $v\in \ell_2(B^{\{a\}})$, then we define
the orthogonal projection $R_{a, v}\in \B(\ell_2(B^{A}))$ onto the 
subspace $\ell_2(B^{A\setminus\{a\}})\otimes \C v$ 
of dimension $|B|^{|A|-1}$ by
$$P_{\alpha, v}^{A, B}=I\otimes ... 
\otimes \overset{a}{R_{v}}\otimes ... \otimes I.$$
\end{definition}

More explicitly 
for each 
  $x=\sum_{t\in B^A} x_te_t\in \ell_2(B^A)$, $a\in A$ and $s\in  B^{A\setminus\{a\}}$
 we define $x(s)=\sum x(s)_{(a, b)}e_{\{(a, b)\}}\in \ell_2(B^{\{a\}})$ by
$$x(s)_{(a, b)}=  x_{s\cup \{(a, b)\}}.\leqno (1)$$
That is we arrange the coordinates of $x$ into $|A|^{|B|-1}$ blocks
$x(s)$ for $s\in B^{A\setminus\{a\}}$.
Then 
given $v=\sum_{b\in B}v^be_{\{(a, b)\}}\in \ell_2(B^{\{a\}})$ we define 
$P_{a, v}^{A, B}: \ell_2(B^A)\rightarrow \ell_2(B^A)$ by
$$P_{a, v}^{A, B}(\sum_{t\in B^A} x_te_t)=
\sum_{s\in B^{A\setminus\{a\}}}\sum_{b\in B}\langle x(s), v\rangle v^be_{s\cup\{(a, b)\}}.\leqno (2)$$
That is to each block $x(s)$ of the coordinates of $x$ we apply the
projection $R_v$ onto the direction of $v$.  To check that this
corresponds to Definition \ref{tensor-projection} one can check this
for basic vectors $e_t=e_{s\cup\{(a, b)\}}$, namely
$$\sum_{b\in B}\langle e_{\{(a,b)\}}, v\rangle v^be_{s\cup\{(a, b)\}}
=\big(R_v(e_{\{(a,b)\}})\big)\otimes e_s.$$
\begin{lemma}\label{products}Let $A, B$ 
be finite sets and let $v_a\in \ell_2(B^{\{a\}})$ be nonzero for each $a\in A$.
Then  for any nonzero choice of $v_a\in \ell_2(B^{\{a\}})$ for $a\in A$
the product $\prod_{a\in A}P_{a, v_a}^{A, B}\leq P_{a, v_a}^{A, B}$
is a nonzero
projection.
\end{lemma}
\begin{proof}
$$\prod_{a\in A}P_{a, v_a}^{A, B}=
\prod_{a\in A} (I\otimes ... \otimes \overset{a}R_{v_a}\otimes ... \otimes I)
=\bigotimes_{a\in A}R_{v_a}.$$
\end{proof}

More explicitly if  $v_a=\sum v_{(a, b)}e_{\{(a, b)\}}$ for $a\in A$ then we consider
$$v=\sum_{t\in B^A}\prod_{a\in A}v_{a, t(a)}e_t\in \ell_2(B^A).$$
It is enough to show that each of the projections $P_{a, v_a}^{A, B}$ for
 $a\in A$ leaves $v$ intact.  Indeed by (1) and (2) we have
 $v(s)=(\prod_{a'\in A\setminus \{a\}}v_{a', s(a')})v_a$, so 
 $$P_{a, v_a}^{A, B}(v)= P_{a, v_a}^{A, B} (\sum_{t\in B^A}\prod_{a\in A}v_{a, t(a)}e_t)=$$
$$=\sum_{s\in B^{A\setminus\{a\}}}
\sum_{b\in B}(\prod_{a'\in A\setminus
 \{a\}}v_{a', s(a')})\langle v_a, v_a\rangle v_{(a, b)}e_{s\cup\{(a, b)\}}=$$
$$=\sum_{t\in B^A}\prod_{a\in A}v_{a, t(a)}e_t=v. $$

\begin{lemma}\label{inclined-subspace}
Suppose that $A, B$ are finite sets such that $|A|=m>1$,  $|B|=d\geq 2^7$ and
$\{x^j: j<n\}$ are  vectors of $\ell_2(B^A)$ for some $n\in\N$. Moreover let us assume that
 $nd^{m-1}<(100/91)^d/8$.
Then for every $a\in A$ there is a nonzero $v_a\in \ell_2(B^{\{a\}})$ such
that $\|P_{v_a, a}^{A, B}(x^j)\|^2\leq(9/10)\|x^j\|^2$ for each $j<n$.
\end{lemma}
\begin{proof}
Fix $A, B, a$ and $\{x^j: j<n\}$ as in the lemma. 
Let $x^j=\sum_{t\in B^A} x^j_te_t$. As in (1) we can write it as
$$x^j=\sum_{s\in B^{A\setminus\{a\}}}x^j(s)\otimes e_s.$$
Apply Lemma \ref{inclined-vector} to the collection 
$\{x^j(s): j<n, \ s\in  B^{A\setminus\{a\}}\}$ 
of cardinality $nd^{m-1}$ and obtain a unit vector
$$v=\sum_{b\in B}v^be_{(a, b)}\in \ell_2(B^{\{a\}})$$
such that 
$$\| R_v(x^j(s))\|^2
 \leq(9/10)\|x^j(s)\|^2$$ 
 for all $s\in B^{A\setminus\{a\}}$ and $j<n$.
 For each $s\in B^{A\setminus\{a\}}$ we have 
 $$\|P_{v, a}^{A, B}(x^j)\|^2=
\|\big(I\otimes ... 
\otimes \overset{a}{R_{v}}\otimes ... \otimes I\big)
(\sum_{s\in B^{A\setminus\{a\}}}x^j(s)\otimes e_s)\|^2=$$
$$=\|\sum_{s\in B^{A\setminus\{a\}}}R_{v}(x^j(s))\otimes e_s\|^2=
\sum_{s\in B^{A\setminus\{a\}}}\|R_{v}(x^j(s))\|^2\leq $$
$$\leq(9/10)\sum_{s\in B^{A\setminus\{a\}}}\|x^j(s)\|^2=(9/10)\|x^j\|^2.$$
\end{proof}

More explicitly using (1) and (2)
$$\|P_{v, a}^{A, B}(x^j)\|^2=
\sum_{s\in B^{A\setminus\{a\}}}\|\sum_{b\in B}\langle x^j(s), v\rangle v^be_{s\cup\{(a, b)\}}\|^2=$$
$$=\sum_{s\in B^{A\setminus\{a\}}}\|R_v(x^j(s))\|^2\leq
(9/10)\sum_{s\in B^{A\setminus\{a\}}}\|x^j(s)\|^2=(9/10)\|x^j\|^2.$$

\section{A family of projections and the pure state}

For this section we  fix  $d:\N\setminus\{0\}\rightarrow\N$ such that
for any $m\in \N$, $m>0$ we have
\begin{itemize}
\item $d(m)\geq 2^7$ ,
\item $32m^2(d(m)^{(2^m)})^2d(m)^{(2^m-1)}<(100/91)^{d(m)}$.
\end{itemize}
Such $d$ can be easily constructed as 
for each $m\in \N$ the polynomial $$p_m(x)=32m^2(x^{(2^m)})^2x^{(2^m-1)}$$
is smaller than the exponential function $(100/91)^x$ for sufficiently big $x\in \R$.

In the rest of this section we will identify $d(m)$
with the set $\{0, ..., d(m)-1\}$.
Define
$$\mathcal O=\bigcup_{m>0} d(m)^{(\seqm)}.$$
Note that the summands of this union are pairwise disjoint as they consist
of functions with different domains $\seqm$ for $m\in \N$. In this section instead of the usual 
$\ell_2=\ell_2(\N)$ we will work with $\ell_2(\mathcal O)$. 
For $m>0$ let $Q_m: \ell_2(\mathcal O)\rightarrow\ell_2(\mathcal O)$ be the orthogonal
projection   onto $\ell_2(d(m)^{(\seqm)})$ considered as a subspace of $\ell_2(\mathcal O)$
consisting of vectors whose coordinates in 
$\mathcal O\setminus d(m)^{(\seqm)}$ are zero. We will be dealing with
algebras $\B_m=Q_m\B(\ell_2(\mathcal O))Q_m$ for $m>0$, they will be identified 
with $\B(\ell_2(d(m)^{(\seqm)}))$.
The projections we will be constructing will be elements of
$$\bigoplus_{m>0}\B_m$$
So for operators $T_m\in \B(\ell_2(d(m)^{(\seqm)}))$ for $m>0$ we will have
$$\bigoplus_{m\in \N} T_m\in \B(\ell_2).$$

\begin{lemma}\label{main-lemma} Let $(e_k: k\in \N)$ be  an orthonormal
basis of $\ell_2(\mathcal O)$ and let $\sigma_m\in \seqm$ for each $m>0$.
For each $m>0$ there is $v_m\in \ell_2(d(m)^{\{\sigma_m\}})$
such that for each $k\in \N$ we have
$$|\langle \big(\bigoplus_{m>0}
P^{\seqm, d(m)}_{\sigma_m, v_{m}}\big)(e_k), e_k\rangle|\leq 19/20.$$
\end{lemma}
\begin{proof} 
For $m>0$ let 
$X_m=\{k\in \N: \|Q_m(e_k)\|^2>3/\pi^2m^2\}.$
As $\sum_{m>0}{1\over {m}^2}={\pi^2\over 6}$, note that  for every $k\in \N$ we have
$$\sum_{\{m>0: k\not\in  X_m\}}\|Q_m(e_k)\|^2
\leq {3\over \pi^2}\sum_{m>0}{1\over {m}^2}\leq 1/2. \leqno (*)$$
For $m\in \N$ by Lemma \ref{big-projections}
applied for $\varepsilon= 3/\pi^2m^2$ knowing that the dimension  of
$\B_m$ is $d(m)^{2^m}$ we have that 
$|X_m|\leq \pi^2m^2(d(m)^{2^m})^2/3\leq 4m^2(d(m)^{2^m})^2$.

Now for $m>0$ consider $\{Q_m(e_k): k\in X_m\}$. Since 
$8|X_m|d(m)^{(2^m-1)}<(100/91)^{d(m)}$ using Lemma \ref{inclined-subspace} 
 we can find 
$v_m\in \ell_2(d(m)^{\{\sigma_m\}})$ such that 
$$\|P_{\sigma_m, v_m}^{\seqm, d(m)}(e_k)\|=\|P_{\sigma_m, v_m}^{\seqm, d(m)}(Q_m(e_k))\|
\leq(9/10)\|Q_m(e_k)\|^2$$
for each $k\in X_m$ since 
the range of $P_{\sigma_m, v_m}^{\seqm, d(m)}$ is included in the range of $Q_m$.

For each $k\in \N$ we have
$$\|\bigoplus_{m\in \N}P_{\sigma_m, v_m}^{\seqm, d(m)}(e_k)\|^2\leq
\sum_{\{m>0: k\in X_m\}}\|P_{\sigma_m, v_m}^{\seqm, d(m)}(e_k)\|^2+
\sum_{\{m>0: k\not\in X_m\}}\|Q_{m}(e_k)\|^2\leq$$
$$\leq (9/10)\sum_{\{m>0: k\in X_m\}}\|Q_{m}(e_k)\|^2+
\sum_{\{m>0: k\not\in X_m\}}\|Q_{m}(e_k)\|^2.$$
Putting $\alpha=\sum_{\{m>0: k\not\in X_m\}}\|Q_{m}(e_k)\|^2$ we have 
that $\sum_{\{m>0: k\in X_m\}}\|Q_{m}(e_k)\|^2=1-\alpha$ as $e_k=\sum_{m>0}Q_m(e_k)$. So
$$\|\bigoplus_{m>0}P_{\sigma_m, v_m}^{\seqm, d(m)}(e_k)\|^2\leq (9/10)(1-\alpha)+\alpha$$
However $\alpha\in [0,1/2]$ by (*) and
$(9/10)(1-\alpha)+\alpha=(1/10)\alpha+(9/10)$ assumes its maximum
on $[0,1/2]$ at $\alpha=1/2$. The maximum is  $19/20$ and
so for each $k\in \N$ by Lemma \ref{equality} we have 
$$|\langle \bigoplus_{m>0}P_{\sigma_m, v_m}^{\seqm, d(m)}(e_k),  e_k\rangle|=\|\bigoplus_{m>0}P_{\sigma_m, v_m}^{\seqm, d(m)}(e_k)\|^2
\leq 19/20,$$
as required.

\end{proof}

\begin{theorem}\label{family} There is a collection $(P_\alpha: \alpha\in \{0,1\}^\N)$ of infinite
dimensional orthogonal projections in $\B(\ell_2)$
 such that for any $\alpha_1, ..., \alpha_n\in \{0,1\}^\N$ and $n\in \N$
 there is a nonzero projection $P_{\alpha_1, ...,\alpha_n}\leq P_{\alpha_1}, ..., P_{\alpha_n}$
 and for every
orthonormal basis $(e_k)_{k\in \N}$ of $\ell_2$ there is $\alpha\in \{0,1\}^\N$
such that for each $k\in \N$ we have 
$$|\langle P_\alpha(e_k), e_k\rangle|\leq 19/20.$$
\end{theorem}
\begin{proof} We will prove the theorem for
$\ell_2(\mathcal O)$ instead of $\ell_2=\ell_2(\N)$.
Since $\mathcal O$ is countably infinite, this makes no difference.
Enumerate all orthonormal bases of $\ell_2(\mathcal O)$
as $\{(e_k^\alpha)_{k\in \N}: \alpha\in \{0,1\}^\N\}$. This is possible since
by the cardinal equality $(2^\omega)^\omega=2^\omega$
both $\{0,1\}^\N$ and the collection of all orthonormal bases of $\ell_2$
have the same cardinality equal to the continuum.
For $\alpha\in \{0,1\}^\N$  consider
$$P_\alpha=\bigoplus_{m\in \N}P_{\alpha|m, v_m^\alpha}^{\seqm, d(m)},$$
where $v_m^\alpha$s are chosen according to Lemma \ref{main-lemma}
for the basis $(e^\alpha_k)_{k\in \N}$ and $\sigma_m=\alpha|m$
which is an element of $\{0,1\}^m$ formed by the first $m$ terms of $\alpha$. Hence we have 
$|\langle P_\alpha(e_k^\alpha), e_k^\alpha\rangle|\leq 19/20$ for each $k\in \N$ and
each $\alpha\in \{0,1\}^\N$.

Now let $\alpha_1, ..., \alpha_n\in \{0,1\}^\N$. Let $m\in \N$
be such that $\alpha_j|m\not=\alpha_{j'}|m$ for any two $1\leq j<j'\leq m$.
By Lemma \ref{products}
$$\prod_{1\leq j\leq n} P_{\alpha_j|m, v_m^{\alpha_j}}^{\seqm, d(m)}$$
is a projection dominated by each 
$P_{\alpha_j|m, v_m^{\alpha_j}}^{\seqm, d(m)}$s for $1\leq j\leq m$ 
hence the same is true for
the projections $P_{\alpha_1}, ..., P_{\alpha_n}$.
\end{proof}

To construct our pure state we  a result relating certain collections of projections
in $\B(\ell_2)$ and pure states on $\B(\ell_2)$.
Based on Chapter 6 of \cite{farah-wofsey} it seems that the following result is  due to N. Weaver.
We provide the proof for the convenience of the reader.

\begin{lemma}\label{ext-pure} Suppose that $(P_j)_{j\in J}$ is a collection of projections in
$\B(\ell_2)$ such that for any $j_1, ..., j_n\in J$ there is a nonzero
projection $P$ such that $P\leq P_{j_i}$ for each $1\leq i\leq n$.
Then there is a pure state $\phi$ on $\B(\ell_2)$ such that
$\phi(P_j)=1$ for all $j\in J$.
\end{lemma}
\begin{proof} Let $S$ denote the set of states on $\B(\ell_2)$.
 Let $\mathcal P$ denote the family of all finite subsets of $J$ and
 let $P_a\leq P_{j_1}, ...,  P_{j_n}$ be the projection as in the lemma for
 $a=\{j_1, ...j_n\}$.
  As $\|P_a\|=1$ for every $a\in \mathcal P$, there
are  states $\phi_a\in S$ such that $\phi_a(P_a)=1$ (\cite[5.1.11]{murphy})
which satisfy $\phi(P_j)=1$ for each $j\in a$ as $\phi(P_a)\leq \phi(P_j)\leq 1$.
For $a\in \mathcal P$ consider
$$F_{a}=\{\phi\in S: \phi(P_j)=1\ \hbox{for all}\ j\in a\}$$

$F_{a}$s are convex, nonempty,  weak$^*$ closed and form a centered family
as $F_{a\cup a'}\subseteq F_a\cap F_{a'}$ for all $a, a'\in \mathcal P$, so by the compactness
of the dual ball of $\B(\ell_2)$ in the weak$^*$ topology we have 
$\bigcap_{a\in \mathcal P}F_a\not=\emptyset$. Moreover $\bigcap_{a\in \mathcal P}F_a\not=\emptyset$ is convex as the intersection of convex
sets. By the Krein-Milman theorem $\bigcap_{a\in \mathcal P}F_a$ has an extreme point $\phi$.
We claim that $\phi$ is the desired pure state. 
If $\phi=\alpha \psi+(1-\alpha)\psi'$ for some $\psi, \psi'\in S$ and $\alpha\in (0,1)$,
we would have $\alpha\psi(P_a)+(1-\alpha)\psi'(P_a)=\phi(P_a)=1$ for any $a\in \mathcal P$.
But this implies that $\psi(P_a)=\psi'(P_a)=1$ for all $a\in \mathcal P$, and so
$\psi, \psi\in \bigcap_{a\in \mathcal P}F_a$. However, in such a case,
$\psi=\psi'=\phi$ as $\phi$ was an extreme point of $\bigcap_{a\in \mathcal P}F_a$.
\end{proof}

\begin{theorem}\label{main-theorem} There is a non-diagonalizable pure state in $\B(\ell_2)$.
\end{theorem}
\begin{proof}
Let $(P_\alpha: \alpha\in \{0,1\}^\N)$ be the collection of orthogonal
projections from Theorem \ref{family}.
By Lemma \ref{ext-pure} there is a pure state $\phi$ on $\B(\ell_2)$ such
that $\phi(P_\alpha)=1$ for each $\alpha\in \{0,1\}^\N$. 

However, by Theorem
\ref{family} for every orthonormal basis $(e_k)_{k\in \N}$ of $\ell_2$ there 
is $\alpha\in \{0,1\}^\N$ such that 
$$|\lim_{u}\langle P_\alpha(e_k), e_k\rangle|\leq 19/20\not=1=\phi(P_\alpha)$$
which shows that $\phi$ is not diagonalizable.
\end{proof}

\section{Remarks.}

\subsection{} For any nonprincipal ultrafilter
$u$ on $\N$ one can construct a pure state $\phi$ as in Theorem \ref{main-theorem}
which additionally satisfies $\phi(\bigoplus_{m\in X}Q_m)=1$ for all $X\in u$. This is because
the projections $\bigoplus_{m\in X}Q_m$ can be added to the family of projections
from Theorem \ref{family} maintaining the hypothesis of  Lemma \ref{ext-pure}. It follows that such states can
be multiplicative on a big abelian subalgebras of $\B(\ell_2)$ 
of the form $\mathscr A[K]$ of Section 12.5 of \cite{farah-book}.
Here $\mathscr{A}[K]$ is the von Neumann subalgebra of $\B(\ell_2)$
generated by a pairwise orthogonal collection of finite dimensional 
orthogonal projections in $\ell_2$ whose
supremum is the identity.

I. Farah and N. Weaver  
showed that under  an additional set-theoretic hypothesis 
$ \mathfrak d\leq \mathfrak p^*$ (12.5.10 of \cite{farah-book})
there is a pure state whose restriction to any algebra $\mathscr{A}[K]$
is not multiplicative. I. Farah conjectures (p. 336 of \cite{farah-book})
that it is consistent that any pure state has a multiplicative restriction
to a subalgebra of the form $\mathscr{A}[K]$. So our pure state  is compatible with
 this conjecture.

\subsection{} One can see that the commutators $[P_{\alpha_1}, P_{\alpha_2}]$
of projections from Theorem \ref{family}
for distinct $\alpha_1, \alpha_2\in \{0,1\}^\N$ are finite dimensional, namely they belong to
$\bigoplus_{0<j<m} \B_j$, where $m\in \N$ is minimal such that $\alpha_1|m\not=\alpha_2|m$.
It follows that the image 
under the quotient map in the Calkin algebra of $\{P_\alpha: \alpha\in \{0,1\}^\N\}$
is commutative. 
So it is another example of an uncountable collection of commuting projections
in the Calkin algebra which does not lift to
a simultaneously diagonalizable collection
of projections in $\B(\ell_2)$. Such first examples were constructed by J. Anderson in \cite{pathology}
assuming {\sf CH}. Other constructions that do not require additional
set-theoretic assumptions are based on different combinatorial
arguments than ours: Akemann and Weaver's construction is based on the counting argument (\cite{ch-masas}) 
and Farah's construction is based  on a combinatorial argument due to Luzin (Theorem 14.3.2 of \cite{farah-book}).

\subsection{} A somewhat similar use of finite (but two-fold) tensor products to construct a non-separable
object in $\B(\ell_2)$ was employed in \cite{thin-tall}.

\bibliographystyle{amsplain}

\end{document}